\documentclass[hidelinks,onefignum,onetabnum]{siamart251216}

\usepackage{mathrsfs}
\usepackage{amsfonts}
\usepackage{amssymb}
\usepackage{booktabs}
\usepackage{multirow}
\usepackage{array}
\usepackage{graphicx}
\usepackage{epstopdf}
\usepackage{algpseudocode}
\usepackage{amsopn}

\ifpdf
  \DeclareGraphicsExtensions{.eps,.pdf,.png,.jpg}
\else
  \DeclareGraphicsExtensions{.eps}
\fi

\newsiamremark{assumption}{Assumption}
\crefname{assumption}{assumption}{assumptions}
\Crefname{assumption}{Assumption}{Assumptions}
\newsiamremark{example}{Example}
\crefname{example}{example}{examples}
\Crefname{example}{Example}{Examples}

\algrenewcommand\algorithmicrequire{\textbf{Input:}}
\algrenewcommand\algorithmicensure{\textbf{Output:}}

\headers{Mixed Precision Iterative Refinement for CAREs}{J. Zhang and Y. Zhou}

\title{Newton-Based Mixed Precision Iterative Refinement for Large-Scale Sparse Continuous-Time Algebraic Riccati Equations\thanks{\funding{This work was supported by the National Key Research and Development Program of China (Grant No.~2023YFB3001604).}}}

\author{Juan Zhang\thanks{Key Laboratory of Intelligent Computing and Information Processing of Ministry of Education, 
Hunan Key Laboratory for Computation and Simulation in Science and Engineering, 
School of Mathematics and Computational Science, Xiangtan University, Xiangtan 411105, Hunan, China. 
Corresponding author (\email{zhangjuan@xtu.edu.cn}).}
\and Yang Zhou\footnotemark[2]}

\ifpdf
\hypersetup{
  pdftitle={Newton-Based Mixed Precision Iterative Refinement for Large-Scale Sparse Continuous-Time Algebraic Riccati Equations},
  pdfauthor={Juan Zhang and Yang Zhou}
}
\fi

\begin{document}

\maketitle

\begin{abstract}
We propose a Newton-based mixed precision iterative refinement framework for solving large-scale sparse continuous-time algebraic Riccati equations (CAREs). The framework computes the initial approximation and the inner Lyapunov correction equations in lower precision, while evaluating residuals and updating the solution in higher precision. 
To handle indefinite residuals and Newton correction terms in low-rank form, we introduce factor decomposition procedures with truncation strategies that preserve positive semidefiniteness and control rank growth.
A first-order rounding error analysis derives a residual recurrence for the refinement process and relates stable mixed precision refinement to a Lyapunov operator conditioning threshold governed by the unit roundoff of the lower precision inner solves.
We then present a concrete ADI-based realization, using NLR-ADI for the initial CARE approximation and LR-ADI for the inner Lyapunov correction equations. Compared with dense Lyapunov correction implementations, this realization reduces the main computations to shifted linear solves and low-rank factor operations, and we provide a solver-dependent complexity analysis.
Numerical experiments on dense CARE over a range of condition numbers illustrate the conditioning effect described by the error analysis, and experiments on large-scale sparse CAREs show that the mixed precision framework is faster than the full double precision implementation while maintaining the same level of accuracy.
\end{abstract}



\section{Introduction}
We consider the large-scale continuous-time algebraic Riccati equations
\begin{equation}\label{CARE}
A^T X + X A - XB B^T X + C^T C = 0,
\end{equation}
where $A \in \mathbb{R}^{n \times n}$ is a sparse matrix, $B \in \mathbb{R}^{n \times m}, C \in \mathbb{R}^{p \times n}$ are low-rank matrices satisfy $p,m \ll n$. Such equations arise extensively in optimal control \cite{zuiyoukz2001}, model order reduction \cite{mxjj1983,mxjj2005}, and related fields. For large-scale problems, it is common to seek a low-rank factor approximation $X \approx ZZ^T$ with $\text{rank}(Z) \ll n$. Under the standard assumptions that $(A,B)$ is stabilizable and $(C,A)$ is detectable, equation \eqref{CARE} has a unique symmetric positive semidefinite stabilizing solution $X$, that is, the closed-loop matrix $A-BB^T X$ is stable.

A variety of numerical methods have been developed for algebraic Riccati equations. For small- and medium-scale problems, Schur methods, matrix sign-function methods, and structure-preserving doubling (SDA) algorithms are well-established and generally robust, but their cubic computational cost and quadratic storage requirements become prohibitive as $n$ grows \cite{schur1979,sign1980,sign1998,chu2005doubling}.
For large-scale sparse CAREs, methods typically exploit low-rank solution structures and avoid forming dense matrices. Representative approaches include Krylov subspace and related projection methods \cite{EbArnoldi2009,subspace2015,Krylov2024} and Riccati ADI (RADI) iterations \cite{radi2018}.
Another important class comprises Newton-based inner-outer iterations, which use Newton iteration at the outer level and low-rank solvers for the resulting Lyapunov equations. A typical example is the Newton low-rank ADI (NLR-ADI) method, which uses low-rank ADI (LR-ADI) for the inner Lyapunov equations \cite{INRADI2016,dizhi2016,zhang2024lowrank}.

In recent years, modern high performance computing architectures have provided native hardware support for low precision floating point arithmetic \cite{IEEE}. On many modern architectures, lower precision formats offer higher peak throughput than double precision (FP64) and reduce memory traffic, thereby lowering the cost of data movement.
Consequently, mixed precision algorithm design has attracted increasing attention \cite{mxiedzongsu2021,mxiedzongsu2022}. Such algorithms perform computationally intensive operations in lower precision whenever possible, while retaining higher precision for accuracy sensitive computations. With suitable error control, such mixed precision strategies can provide substantial acceleration while preserving the desired accuracy.
The floating point formats discussed and used in this work are summarized in Table \ref{tab:precision}.

\begin{table}[htbp]
\footnotesize
  \centering
  \caption{Parameters for floating point arithmetics}
  \label{tab:precision}
  \begin{tabular}{lcccc}
    \toprule
    Format & Significand & Exponent & Unit roundoff & Range \\
    \midrule
    FP64      & 53  & 11 & $1.11\times 10^{-16}$  & $10^{\pm308}$  \\
    FP32      & 24  & 8  & $5.96\times 10^{-8}$  & $10^{\pm38}$  \\
    FP16      & 11  & 5  & $4.88\times 10^{-4}$  & $10^{\pm5}$  \\
    Bfloat16  & 8   & 8  & $3.91\times 10^{-3}$  & $10^{\pm38}$  \\
    \bottomrule
  \end{tabular}
\end{table}

Mixed precision techniques have been extensively studied in numerical linear algebra, particularly for linear systems, least squares problems, eigensolvers, and matrix factorizations \cite{carson2018,carson2020ls,prikopa2013,mpqr2025}, but their use for matrix equations remains comparatively limited. Initial explorations have primarily focused on linear matrix equations. These studies include mixed precision Schur algorithms for Sylvester equations \cite{liu2025sylvester} and several developments for Lyapunov equations, including refinement frameworks with rigorous rounding error analysis \cite{liu2025lyapunov} and recent mixed precision variants of low-rank ADI methods \cite{arXiv2026Lyapunov}.
For the more challenging nonlinear CAREs, Benner et al. \cite{benner2017,benner2018} presented an early mixed-precision Newton framework. Their work focused primarily on hybrid CPU-GPU execution and energy efficiency. 
In their implementation, the initial approximation was computed using a low-rank SDA method, whereas the Lyapunov correction equations were solved using the matrix sign-function method. The dense inverse-type operations required by the latter limit the direct exploitation of sparsity in large-scale CAREs.

Motivated by these considerations, we propose a general mixed precision Newton refinement framework for large-scale sparse CAREs. Our main contributions are as follows:

\begin{itemize}
    \item \textbf{A structure preserving mixed precision Newton framework:}
    We propose a mixed precision Newton refinement framework for CARE. The framework separates lower precision initial CARE solves and inner Lyapunov solves from higher precision residual evaluation and solution updates. The residual and solution factor decompositions are designed to preserve positive semidefiniteness of the computed solution iterates while controlling rank growth.

    \item \textbf{First-order rounding error analysis:}
    We develop a first-order rounding error analysis for the proposed mixed precision refinement process. The analysis derives perturbation bounds for the residual factor decomposition and the solution factor update, and establishes a residual recurrence that characterizes how the inner solve precision, factorization errors, and truncation errors affect the attainable accuracy. This identifies conditions under which the residual recurrence is contractive up to a precision dependent accuracy floor.

    \item \textbf{Low-rank ADI realization and complexity analysis:}
    We specialize the general framework to large-scale sparse CAREs by using NLR-ADI for the initial approximation and LR-ADI for the inner Lyapunov correction equations. The resulting ADI-based implementation keeps the main computations in shifted linear solves and low-rank factor operations, rather than dense inverse-type operations as in sign function correction approaches. We also give a complexity analysis that identifies the dominant costs and explains why this implementation is suitable for large-scale sparse problems.
  \end{itemize}

  The rest of the paper is organized as follows. Section 2 presents the mixed-precision Newton refinement framework and the associated low-rank factor decompositions. Section 3 develops the first-order rounding error analysis. Section 4 describes the ADI-based realization and complexity analysis for large-scale sparse CAREs. Section 5 reports numerical experiments on synthetic test problems and SuiteSparse benchmark matrices, and Section 6 concludes the paper.

  \section{The Mixed Precision Iterative Refinement Framework}

We first recall the standard Newton iteration for \eqref{CARE}. Define the Riccati operator
$$ \mathcal{R}(X) = A^T X + X A - XBB^T X + C^T C.$$
Starting from an initial approximation $X_0$, the Newton iteration generates iterates $ X_{k+1}=X_k+Y_k$, where the correction $Y_k$ satisfies $\mathcal{R}'_{X_k}(Y_k)=-\mathcal{R}(X_k)$. Here $\mathcal{R}'_{X_k}$ denotes the Fr\'echet derivative of $\mathcal{R}$ at $X_k$.

The Fr\'{e}chet derivative of $\mathcal{R}$ at $X$ applied to a direction $Y$ has the Lyapunov form
\begin{equation*}
\mathcal{R}'_X[Y] = (A - BB^T X)^T Y + Y(A - BB^T X).
\end{equation*}
Therefore, at the $k$-th Newton step, the correction $Y_k$ is obtained by solving the Lyapunov equation
\begin{equation}\label{lyap_eq}
(A - BB^T X_k)^T Y_k + Y_k (A - BB^T X_k) = - \mathcal{R}(X_k).
\end{equation}
If the Lyapunov equation \eqref{lyap_eq} is solved exactly at each step, the resulting sequence of symmetric positive semidefinite matrices $\{X_k\}_{k\geq 0}$ satisfies $X_k \geq X_{k+1}$ for all $k \geq 1$. Furthermore, the matrix $A - BB^TX_k$ remains stable, and the sequence $\{X_k\}_{k\geq 0}$ converges quadratically to the stabilizing solution $X$ \cite{projectNK2019}.

The Newton iteration for the CARE has a correction and update structure analogous to iterative refinement for linear systems. This motivates a mixed precision Newton refinement framework.

To facilitate the algorithmic description and subsequent analysis, we first define the precision levels used throughout the paper. Let $u_s$, $u_r$, $u$, and $u_c$ denote the unit round offs used in different phases of the framework. Specifically, $u_s$ is used for the low precision initial CARE solve and inner Lyapunov solves, $u_r$ for residual evaluation, $u$ for working/storage precision, and $u_c$ for the solution update.

To solve \eqref{lyap_eq} efficiently in the large-scale setting, we use a low-rank formulation. Since the approximate CARE solution is represented as $X_k \approx Z_kZ_k^T$, the corresponding residual can be expressed through low-rank factors and used as the right hand side of the correction Lyapunov equations. Thus, \eqref{lyap_eq} can be solved by a low-rank Lyapunov solver, such as the LR-ADI algorithm. The CARE residual evaluated at $Z_kZ_k^T$ is
\begin{equation}\label{R}
\mathcal{R}(Z_k) = A^T Z_k Z_k^T + Z_k Z_k^T A - Z_k Z_k^T B B^T Z_k Z_k^T + C^T C.
\end{equation}
In general, this residual is indefinite. Therefore, we represent it as the difference of two symmetric positive semidefinite matrices
$$
\mathcal{R}(Z_k) = \mathcal{R}^+(Z_k) - \mathcal{R}^-(Z_k) = C_k^+(C_k^+)^T - C_k^-(C_k^-)^T.
$$
Subsequently, the correction factors $Z_k^+$ and $Z_k^-$ are determined by solving the following two Lyapunov equations
\begin{equation}\label{lyap2}
\begin{aligned}
A_k^T X_k^+ &+ X_k^+A_k = -C^+_k (C^+_k)^T , \quad X_k^+ = Z_k^+ (Z_k^+)^T,\\
A_k^TX_k^- &+ X_k^-A_k  = -C^-_k (C^-_k)^T , \quad X_k^- = Z_k^- (Z_k^-)^T,
\end{aligned}
\end{equation}
where $A_k = A - BB^T X_k$.

Directly computing this decomposition from the full residual $\mathcal{R}(Z_k)$ is impractical for large $n$. Let $r_k$ denote the number of columns of $Z_k$. From the structure of \eqref{R}, the residual can be written in the factored form
$$
\begin{aligned}
\mathcal{R}(Z_k) &= A^T Z_k Z_k^T + Z_k Z_k^T A - Z_k Z_k^T B B^T Z_k Z_k^T + C^T C\\
&  = [Z_k,A^T Z_k , Z_k Z_k^T B ,C^T] \left[\begin{array}{cccc} 0 & I_{r_k} & 0 & 0\\ I_{r_k} & 0 & 0 & 0 \\ 0 & 0 & -I_m & 0  \\ 0 & 0 & 0 & I_p\\ \end{array} \right] \left[\begin{array}{c} Z_k \\A^T Z_k \\ Z_k Z_k^T B \\ C^T \\ \end{array} \right]^T \\
& =: F_kM_kF_k^T,
\end{aligned}
$$
where $F_k \in \mathbb{R}^{n \times q_k}$, $q_k := 2r_k + m + p \ll n$. Compute a thin QR decomposition $F_k=U_kT_k$, where the columns of $U_k$ form an orthonormal basis for $\operatorname{range}(F_k)$. Then
$$
\mathcal{R}(Z_k)=U_k H_k U_k^T,\qquad H_k:=T_kM_kT_k^T.
$$
Since $H_k$ has dimension equal to the column dimension of $F_k$, which is much smaller than $n$, we compute the eigenvalue decomposition
$$H_k = [Q_k^+,Q_k^-]\left[\begin{array}{cc} \Lambda_k^+ & 0\\ 0 & \Lambda_k^- \\ \end{array} \right][Q_k^+,Q_k^-]^T.$$
Let $J^+ := \{ j \mid \lambda_j > 0\}$ and $J^- := \{ j \mid \lambda_j \leq 0\}$, where $\lambda_j$ are the eigenvalues of $H_k$. We define $\Lambda_k^+ = \text{diag}(\lambda_j)_{j \in J^+}$ and $\Lambda_k^- = \text{diag}(\lambda_j)_{j \in J^-}$, with $Q_k^+$ and $Q_k^-$ denoting the corresponding eigenvector matrices.

The positive and negative residual factors are then given by
$$C_k^+ = U_k Q_k^+ (\Lambda_k^+)^{1/2} \quad \text{and} \quad C_k^- = U_k Q_k^- (-\Lambda_k^-)^{1/2}.$$
In practice, to enhance numerical robustness, we truncate negligible eigenvalues based on a relative threshold $\eta_r > 0$. This systematic decomposition procedure is summarized in Algorithm \ref{res_fac1}.

\begin{algorithm}
\caption{Residual Factor Decomposition}
\label{res_fac1}
\begin{algorithmic}[1]
\Require Matrices $A, B, C$, solution factors $Z_k$, relative threshold parameter $\eta_r>0$;
\State Let $F_k = [Z_k,A^T Z_k , Z_k Z_k^T B ,C^T]$, compute QR decomposition of $F_k$: $F_k = U_k T_k$;
\State Let $H_k = T_k M_k T_k^T$, compute spectral decomposition of $H_k$: $H_k = Q_k \Lambda_k Q_k^T$;
\State $\Lambda_k^+ = \text{diag}(\lambda_j),j \in J^+:=\{ j | \lambda_j \geq \eta_r\}, Q_k^+ = Q_k(:,J^+)$;
\State $\Lambda_k^- = \text{diag}(\lambda_j),j \in J^-:=\{ j | \lambda_j \leq -\eta_r\}, Q_k^- = Q_k(:,J^-)$;
\State $C_k^+ = U_k Q_k^+ (\Lambda_k^+)^{1/2}$, $C_k^- = U_k Q_k^- (-\Lambda_k^-)^{1/2}$;
\Ensure Residual factors $C_k^+,C_k^-$.
\end{algorithmic}
\end{algorithm}

In exact arithmetic, after solving the correction equations \eqref{lyap2}, the updated solution has the form $X_{k+1}=X_k+X_k^+-X_k^-$, where both $X_k^+$ and $X_k^-$ are symmetric positive semidefinite. In floating point arithmetic, however, the subtraction of $X_k^-$ may introduce cancellation errors, and the computed update may lose positive semidefiniteness. To preserve the positive semidefinite structure of the solution iterate, we perform a solution factor decomposition and retain only the positive spectral part.

Let $r_k^+$ and $r_k^-$ denote the numbers of columns of $\widehat{Z}_k^+$ and $\widehat{Z}_k^-$, respectively, and set $s_k = r_k + r_k^+ + r_k^-$. Using the same low-rank decomposition strategy as for the residual, we express the updated matrix $X_{k+1} = Z_k Z_k^T + Z_k^+ (Z_k^+)^T - Z_k^- (Z_k^-)^T$ as
$$
X_{k+1} = [Z_k,Z_k^+,Z_k^-]\left[\begin{array}{ccc} I_{r_k} & 0 & 0\\ 0 & I_{r_k^+} & 0 \\ 0 & 0 & -I_{r_k^-} \\ \end{array} \right] [Z_k,Z_k^+,Z_k^-]^T := G_k N_k G_k^T,
$$
where $G_k \in \mathbb{R}^{n \times s_k}$. We compute a thin QR decomposition $G_k = V_k \Gamma_k$, which allows us to write $X_{k+1} = V_k P_k V_k^T$, where $P_k := \Gamma_k N_k \Gamma_k^T$. We then compute the eigenvalue decomposition $P_k = \Theta_k \Sigma_k \Theta_k^T$. The positive semidefinite part of $P_k$ is obtained by retaining its positive eigenvalues. Therefore, the updated solution factor is given by $Z_{k+1} := V_k \Theta_k^+ (\Sigma_k^+)^{1/2}$, where $\Theta_k^+$ contains the corresponding eigenvectors, and $\Sigma_k^+ = \text{diag}(\sigma_j)$ for $j \in J^+ := \{ j \mid \sigma_j > 0\}$.
Moreover, $P_k^-=\Theta_k^-\Sigma_k^-(\Theta_k^-)^T$ denotes the negative spectral part of $P_k$.

To control rank growth in practice, we introduce a relative truncation tolerance $\eta_s>0$ and retain only positive eigenvalues above the prescribed relative threshold, following the same strategy as in the residual decomposition. The resulting solution factor decomposition is summarized in Algorithm \ref{sol_fac}.

\begin{algorithm}[htbp]
\caption{Solution Factor Decomposition}
\label{sol_fac}
\begin{algorithmic}[1]
\Require Solution factors  $Z_k$, $Z_k^+, Z_k^-$, relative threshold parameter $\eta_s>0$;
\State Let $G_k = [Z_k,Z_k^+,Z_k^-]$, compute QR decomposition of $G_k$: $G_k = V_k \Gamma_k$;
\State Let $P_k = \Gamma_k N_k \Gamma_k^T$, compute spectral decomposition of $P_k$: $P_k = \Theta_k \Sigma_k \Theta_k^T$;
\State $\Sigma_k^+ = \text{diag}(\sigma_j),j \in J^+:=\{ j | \sigma_j \geq \eta_s \}, \Theta_k^+ = \Theta_k(:,J^+)$;
\State $Z_{k+1} = V_k \Theta_k^+ (\Sigma_k^+)^{1/2}$;
\Ensure Solution factors $Z_{k+1}$.
\end{algorithmic}
\end{algorithm}

Framework \ref{IR_1} summarizes the complete mixed precision iterative refinement framework for the CAREs. The algorithm terminates when the residual Frobenius norm $\|\mathcal{R}(Z_k)\|_F$ falls below a prescribed tolerance. This stopping criterion introduces negligible additional cost, since the norm can be evaluated from the eigenvalues of the small core matrix computed before truncation in Algorithm \ref{res_fac1}.

\floatname{algorithm}{Framework}
\begin{algorithm}[!htb]
\caption{Factorization-based Mixed Precision Newton Iterative Refinement Framework}
\label{IR_1}
\begin{algorithmic}[1]
\Require Matrices $A,B,C,\tau_R,k_{\max}$, precisions $u_s \geq u \geq u_r, u_c >0$;
\State Compute $A^TX + XA - XBB^TX + C^TC = 0$ in precision $u_s$, obtain approximate solution factor $Z_1$, and store $Z_1$ in precision $u$;
\For{$k=1,2,\dots,k_{\max}$}
\State Apply Algorithm \ref{res_fac1} to decompose residual factors in precision $u_r$, and store $C_k^+, C_k^-$ in precision $u_s$;
\If{$\|\mathcal{R}(Z_{k}Z_{k}^T)\|_F \leq \tau_R\|C^TC\|_F$}
\State {\bf break};
\EndIf
\State Solve $A_k^T X_k^+ + X_k^+ A_k + C^+_k (C^+_k)^T  = 0$ and $A_k^T X_k^- + X_k^-A_k + C^-_k (C^-_k)^T  = 0$ in precision $u_s$, store solution factors $Z_k^+, Z_k^-$ in precision $u$;
\State Apply Algorithm \ref{sol_fac} to update solution factor in precision $u_c$ and obtain $Z_{k+1}$;
\EndFor
\Ensure Solution factor $Z_{k+1}$, such that $X_{k+1} \approx Z_{k+1}Z_{k+1}^T$.
\end{algorithmic}
\end{algorithm}
\floatname{algorithm}{Algorithm}

Although the factorization derivations in the preceding sections were formulated independently of the arithmetic precision, their numerical implementation in Framework \ref{IR_1} is subject to a specific precision hierarchy. Following the standard iterative refinement paradigm, both the computation of the initial solution and the inner Lyapunov solves are performed at the lowest precision $u_s$. To recover the low order digits lost during this solve phase, the subsequent residual evaluation and solution update must be performed at higher precision. Hence, we enforce the hierarchy
\begin{equation*}
u_s \geq u \geq u_r, u_c > 0,
\end{equation*}
where a larger unit roundoff indicates lower arithmetic precision.

\section{Rounding Error Analysis}

In this section, we derive the rounding error bounds for Framework \ref{IR_1}. To facilitate this analysis, we first establish the notation for floating point arithmetic and error bounds. To distinguish the precision levels within our Framework \ref{IR_1}, we attach subscripts to $u$ corresponding to the specific algorithmic phases defined in Section 2.

Following the standard floating-point error model, for any positive integer $n$ satisfying $cnu<1$, we define the constants
$$ \gamma_n = \frac{nu}{1 - nu}, \quad \text{and} \quad \widetilde{\gamma}_n = \frac{cnu}{1 - cnu}, $$
where $c$ is a small integer constant depending on the specific matrix operation. These constants conveniently encapsulate the first-order accumulation of rounding errors in basic matrix computations. When $\gamma_n$ or $\widetilde{\gamma}_n$ carries a superscript, it indicates that the underlying unit roundoff is evaluated at the specified precision level; for example, $\gamma_n^s := \dfrac{nu_s}{1 - nu_s}$.

Throughout the first-order rounding error analysis, we use $\widehat{\cdot}$ to denote quantities computed in finite precision; for example, $\widehat{Z}_k$ is the computed approximation of the exact factor $Z_k$, and $\widehat{X}_k=\widehat{Z}_k\widehat{Z}_k^T$ is the corresponding computed solution approximation. For a low-rank factor $Z$, we write $\mathcal{R}(Z)$ as shorthand for $\mathcal{R}(ZZ^T)$. For a matrix $M$, $|M|$ denotes the componentwise absolute value, and matrix inequalities are interpreted componentwise. The symbol $\lesssim$ denotes an inequality up to a modest constant factor, with higher-order terms in the unit roundoffs neglected.

With these notations in place, we proceed with the first-order rounding error analysis of Framework \ref{IR_1} under the following standard assumptions.

\begin{assumption}\label{ass_factor}
The computed residual and solution factors are treated as exact low-rank representations of the corresponding computed symmetric positive semidefinite matrices. That is, for the residual splitting,
$$
\widehat{\mathcal{R}}^+(\widehat{Z}_k)
= \widehat{C}_k^+(\widehat{C}_k^+)^T,\quad
\widehat{\mathcal{R}}^-(\widehat{Z}_k)
= \widehat{C}_k^-(\widehat{C}_k^-)^T,
$$
and for the solution factors,
$$
\widehat{X}_k=\widehat{Z}_k\widehat{Z}_k^T,\quad
\widehat{X}_k^+=\widehat{Z}_k^+(\widehat{Z}_k^+)^T,\quad
\widehat{X}_k^-=\widehat{Z}_k^-(\widehat{Z}_k^-)^T.
$$
\end{assumption}

\begin{assumption}\label{ass_ly}
Let $\mathcal{L}_k(X) = A_k^T X + XA_k$ denote the Lyapunov operator. The computed solution factors $\widehat{Z}_k^+$ and $\widehat{Z}_k^-$ generated by the inner solver (Line 7 of Framework \ref{IR_1}) satisfy the perturbed equations
\[
A_k^T \widehat{X}_k^+ + \widehat{X}_k^+ A_k + C^+_k (C^+_k)^T = 0, \quad A_k^T \widehat{X}_k^- + \widehat{X}_k^- A_k + C^-_k (C^-_k)^T  = 0,
\]
such that the combined backward error is bounded by
\begin{equation}\label{solver}
\begin{aligned}
\|\mathcal{L}_k(\widehat{X}_k^+ - \widehat{X}_k^-)
+ \widehat{\mathcal{R}}^+(\widehat{Z}_k)
- \widehat{\mathcal{R}}^-(\widehat{Z}_k)\|_F
&\leq u_s d_1 \|\mathcal{L}_k\|_F \|\widehat{X}_k^+ - \widehat{X}_k^-\|_F \\
&\qquad + u_s d_2 \|\widehat{\mathcal{R}}^+(\widehat{Z}_k) - \widehat{\mathcal{R}}^-(\widehat{Z}_k)\|_F .
\end{aligned}
\end{equation}
where $d_1, d_2$ are modest constants depending on the problem dimension and the specific solver employed.
\end{assumption}

Assumption \ref{ass_ly} formalizes the backward stability of the inner solver. Provided that the precision $u_s$ is sufficiently small (satisfying $d_1\kappa(\mathcal{L}_k)u_s < 1$) and the solver is numerically stable, this assumption ensures that the perturbations introduced by the low precision inner solves are proportional to the residuals and bounded by $\mathcal{O}(u_s)$. 
This condition is essential for the convergence analysis of Framework \ref{IR_1}, since it allows the errors from the lower precision inner solves to be controlled by the higher precision residual evaluation and solution updates.

Under the standard floating point model, the computed matrix $\widehat{F}_k$ satisfies
$$\widehat{F}_k = [\widehat{Z}_k,A^T\widehat{Z}_k+\Delta F_k^{(2)},\widehat{Z}_k\widehat{Z}_k^TB +\Delta F_k^{(3)}, C^T],$$
where the rounding error $\Delta F_k^{(2)}$ satisfies $|\Delta F_k^{(2)}|\leq \gamma_n^r|A^T||\widehat{Z}_k|$, and $\Delta F_k^{(3)}$ satisfies $|\Delta F_k^{(3)}|\leq (\gamma_n^r+\gamma_{r_k}^r)|\widehat{Z}_k||\widehat{Z}_k^T||B|$.

Introduce the stability constants $b_1$, $b_2$, and $b_3$ by
\begin{align}
\||A^T||\widehat{Z}_k|\|_F + \||\widehat{Z}_k||\widehat{Z}_k^T||B|\|_F &= b_1\|F_k\|_F, \label{b1}\\
\|T_kT_k^T\|_F &= b_2\|T_kM_kT_k^T\|_F = b_2\|H_k\|_F, \label{b2}\\
\|\Gamma_k\Gamma_k^T\|_F &= b_3\|\Gamma_kN_k\Gamma_k^T\|_F = b_3\|P_k\|_F. \label{b3}
\end{align}
Here $b_1$ measures the normwise growth in forming $A^T\widehat{Z}_k$ and $\widehat{Z}_k(\widehat{Z}_k^TB)$, while $b_2$ and $b_3$ measure possible numerical cancellation in the formation of $H_k$ and $P_k$, respectively.

We define the computed residual and the associated residual errors by
\begin{align}
\widehat{\mathcal{R}}(\widehat{Z}_k) &:= \widehat{\mathcal{R}}^+(\widehat{Z}_k) - \widehat{\mathcal{R}}^-(\widehat{Z}_k) + \Delta\mathcal{R}_k^s, \label{computed_residual_def}\\
\Delta\widehat{\mathcal{R}}_k &:= \widehat{\mathcal{R}}^+(\widehat{Z}_k) - \widehat{\mathcal{R}}^-(\widehat{Z}_k) - \mathcal{R}(\widehat{Z}_k), \label{factorized_residual_error}\\
\Delta\mathcal{R}(\widehat{Z}_k) &:= \widehat{\mathcal{R}}(\widehat{Z}_k) - \mathcal{R}(\widehat{Z}_k).
\label{full_residual_error}
\end{align}
where $\| \Delta \mathcal{R}_k^s\|_F \leq 2u_s \|\widehat{\mathcal{R}}(\widehat{Z}_k)\|_F$. These definitions imply
\begin{equation}\label{residual_error_relation}
\Delta\mathcal{R}(\widehat{Z}_k) = \Delta\widehat{\mathcal{R}}_k + \Delta\mathcal{R}_k^s.
\end{equation}
The solution update error is defined by
\begin{equation}\label{solution_update_error}
\Delta E_k := \widehat{X}_{k+1} - ( \widehat{X}_k+\widehat{X}_k^+-\widehat{X}_k^- ).
\end{equation}

The residual factor decomposition and the solution factor update in Framework \ref{IR_1} have the same low-rank spectral decomposition structure as the procedures analyzed in \cite{liu2025lyapunov}. 
Adapting the first-order normwise perturbation analysis in \cite{liu2025lyapunov} to the factor pairs $(F_k,H_k)$ and $(G_k,P_k)$ in our setting gives the following error bounds.

\begin{lemma}\label{lem_factor_update}
Under the standard floating-point model, and neglecting terms of second and
higher order in the unit roundoffs and truncation thresholds, the residual
factor decomposition and the low-rank solution update at step $k$ of
Framework~\ref{IR_1} satisfy
\begin{equation}\label{alphabeta}
\begin{aligned}
\|\Delta\widehat{\mathcal{R}}_k\|_F
&\lesssim \left( 2\sqrt{m_k}\eta_r + \left( 2b_1b_2+(4b_2+2)m_k^{3/2} \right)\widetilde{\gamma}_n^r \right)
\|\mathcal{R}(\widehat{Z}_k)\|_F\\
&=: \bar{\alpha}_k \|\mathcal{R}(\widehat{Z}_k)\|_F,
\\
\|\Delta\mathcal{R}(\widehat{Z}_k)\|_F
&\lesssim \left( 2\sqrt{m_k}\eta_r + 2u_s + \left( 2b_1b_2+(4b_2+2)m_k^{3/2} \right)\widetilde{\gamma}_n^r \right)
\|\mathcal{R}(\widehat{Z}_k)\|_F\\
&=: \alpha_k \|\mathcal{R}(\widehat{Z}_k)\|_F,
\\
\|\Delta E_k\|_F
&\lesssim \left( 2\sqrt{p_k}\eta_s +\eta_n +(4b_3+2)p_k^{3/2}\widetilde{\gamma}_n^c \right)
\|\widehat{X}_{k+1}\|_F\\
&=: \beta_k \|\widehat{X}_{k+1}\|_F.
\end{aligned}
\end{equation}
Here,
$$
m_k=4\max\{m,p,r_k\},
\qquad
p_k=3\max\{r_k,r_k^+,r_k^-\},
$$
$\eta_n$ is determined by
$$
\|P_k^-\|_F=\eta_n\|P_k\|_F
$$
and is assumed to satisfy $\eta_n \ll 1$.
In particular,
$$
\alpha_k=\bar{\alpha}_k+2u_s.
$$
\end{lemma}

Combining  Assumptions \ref{ass_factor} and \ref{ass_ly}  with Lemma \ref{lem_factor_update}, we obtain the following backward error recurrence for Framework \ref{IR_1}.

\begin{theorem}\label{thm_res}
Let $\widehat{X}_k=\widehat{Z}_k\widehat{Z}_k^T$ be the computed iterate generated by Framework~\ref{IR_1}, and suppose that the closed-loop matrix $A_k=A-BB^T\widehat{X}_k$ is stable. Assume that the inner Lyapunov solver satisfies Assumption~\ref{ass_ly}, and let $\alpha_k$ and $\beta_k$ be defined as in Lemma~\ref{lem_factor_update}.
For all $k \geq 0$, if
$$ d_1\kappa_F(\mathcal{L}_k)u_s<1, $$
define
$$ c_1 = u_s \dfrac{d_1\kappa_F(\mathcal{L}_k)+d_2}{1-u_sd_1\kappa_F(\mathcal{L}_k)}, $$
and
$$
\begin{aligned}
\rho_k &= c_1(1+\bar{\alpha}_k)+\bar{\alpha}_k,\\
\mu_k &= 2(1+c_1)^2(1+\bar{\alpha}_k)^2 \|B\|_F^2 \|\mathcal{L}_k^{-1}\|_F^2,\\
\nu_k &= 2\beta_k^2\|B\|_F^2.
\end{aligned}
$$
Then the residual associated with the computed next factor $\widehat{Z}_{k+1}$ satisfies
\begin{equation}
\begin{aligned}
\|\mathcal{R}(\widehat{Z}_{k+1})\|_F &\lesssim \rho_k \|\mathcal{R}(\widehat{Z}_k)\|_F
+ \mu_k \|\mathcal{R}(\widehat{Z}_k)\|_F^2  \\
&\quad + \beta_k\|\mathcal{L}_k\|_F\|\widehat{X}_{k+1}\|_F + \nu_k\|\widehat{X}_{k+1}\|_F^2.
\end{aligned}
\end{equation}
\end{theorem}

\begin{proof}

By the definition of the solution update error, the computed iterate satisfies
$$
\widehat{X}_{k+1} = \widehat{X}_{k} + \widehat{X}_{k}^+ - \widehat{X}_{k}^- + \Delta E_k, 
$$
Substituting this expression into $\mathcal{R}(\widehat{X}_{k+1})$ and using $\widehat{X}_k=\widehat{Z}_k\widehat{Z}_k^T$, 
together with the definition of the Lyapunov operator $\mathcal{L}_k$, a direct expansion gives
\begin{equation}\label{residual_expansion}
\begin{aligned}
\mathcal{R}(\widehat{Z}_{k+1}) &= \mathcal{R}(\widehat{Z}_{k}) + \mathcal{L}_k(\widehat{X}_{k}^+ - \widehat{X}_{k}^-) + \mathcal{L}_k(\Delta E_k) \\
&\quad - (\widehat{X}_{k}^+ - \widehat{X}_{k}^- + \Delta E_k) BB^T (\widehat{X}_{k}^+ - \widehat{X}_{k}^- + \Delta E_k).
\end{aligned}
\end{equation}
Let 
\begin{equation}\label{inner_solver_error}
  W_k:=\mathcal{L}_k(\widehat{X}_k^+ - \widehat{X}_k^-) + (\widehat{\mathcal{R}}^+(\widehat{Z}_k) - \widehat{\mathcal{R}}^-(\widehat{Z}_k))
\end{equation}
denote the backward error introduced by the solver when computing the solution increment at step $k$. Combining \eqref{factorized_residual_error} and \eqref{inner_solver_error} gives
\begin{equation}\label{lyapunov_increment_relation}
\mathcal{L}_k ( \widehat{X}_k^+-\widehat{X}_k^- ) =  W_k - (\mathcal{R}(\widehat{Z}_k) +  \Delta\widehat{\mathcal{R}}_k ).
\end{equation}
Substituting \eqref{lyapunov_increment_relation} into
\eqref{residual_expansion}, we obtain
\begin{equation}\label{final_residual_identity}
\begin{aligned}
\mathcal{R}(\widehat{Z}_{k+1})
={}& W_k - \Delta\widehat{\mathcal{R}}_k + \mathcal{L}_k(\Delta E_k)\\
&- ( \widehat{X}_k^+-\widehat{X}_k^-+\Delta E_k ) BB^T ( \widehat{X}_k^+-\widehat{X}_k^-+\Delta E_k ).
\end{aligned}
\end{equation}

By Lemma \ref{lem_factor_update}, the residual factorization error and the
solution update error satisfy
\begin{equation}\label{factor_update_bounds}
\|\Delta\widehat{\mathcal{R}}_k\|_F \lesssim \bar{\alpha}_k \|\mathcal{R}(\widehat{Z}_k)\|_F,
\qquad \|\Delta E_k\|_F \lesssim \beta_k \|\widehat{X}_{k+1}\|_F.
\end{equation}
Moreover, by the definition of $\Delta\widehat{\mathcal{R}}_k$,
$$
\begin{aligned}
\| \widehat{\mathcal{R}}^+(\widehat{Z}_k) - \widehat{\mathcal{R}}^-(\widehat{Z}_k) \|_F
&\leq \|\mathcal{R}(\widehat{Z}_k)\|_F + \|\Delta\widehat{\mathcal{R}}_k\|_F\\
&\lesssim (1+\bar{\alpha}_k) \|\mathcal{R}(\widehat{Z}_k)\|_F.
\end{aligned}
$$

From Assumption \ref{ass_ly}, we have
\begin{equation}\label{solver_error_bound}
\|W_k\|_F \leq u_s d_1 \|\mathcal{L}_k\|_F \|\widehat{X}_k^+-\widehat{X}_k^-\|_F + u_s d_2 \|
\widehat{\mathcal{R}}^+(\widehat{Z}_k) - \widehat{\mathcal{R}}^-(\widehat{Z}_k) \|_F.
\end{equation}
Since $\mathcal{L}_k$ is invertible, the definition of $W_k$ gives
\begin{equation}\label{increment_bound_initial}
\|\widehat{X}_k^+-\widehat{X}_k^-\|_F \leq
\|\mathcal{L}_k^{-1}\|_F ( \|W_k\|_F + \| \widehat{\mathcal{R}}^+(\widehat{Z}_k) - \widehat{\mathcal{R}}^-(\widehat{Z}_k) \|_F).
\end{equation}
Substituting \eqref{increment_bound_initial} into \eqref{solver_error_bound} and rearranging yields
\begin{equation}\label{W_split_bound}
\begin{aligned}
\|W_k\|_F & \leq u_s \frac{ d_1\kappa_F(\mathcal{L}_k)+d_2 }{ 1-u_s d_1\kappa_F(\mathcal{L}_k) } 
\| \widehat{\mathcal{R}}^+(\widehat{Z}_k) - \widehat{\mathcal{R}}^-(\widehat{Z}_k) \|_F\\
& \lesssim c_1(1+\bar{\alpha}_k) \|\mathcal{R}(\widehat{Z}_k)\|_F,
\end{aligned}
\end{equation}
where $c_1 = u_s \dfrac{ d_1\kappa_F(\mathcal{L}_k)+d_2 }{ 1-u_s d_1\kappa_F(\mathcal{L}_k) }$. Thus, we have 
\begin{equation*}
\|\widehat{X}_k^+-\widehat{X}_k^-\|_F
\lesssim (1+c_1)(1+\bar{\alpha}_k) \|\mathcal{L}_k^{-1}\|_F\|\mathcal{R}(\widehat{Z}_k)\|_F,
\end{equation*}
which implies that
\begin{equation*}
\| \widehat{X}_k^+ - \widehat{X}_k^- + \Delta E_k \|_F
\lesssim (1+c_1)(1+\bar{\alpha}_k) \|\mathcal{L}_k^{-1}\|_F
\|\mathcal{R}(\widehat{Z}_k)\|_F + \beta_k \|\widehat{X}_{k+1}\|_F.
\end{equation*}

Squaring both sides and applying standard algebraic inequalities, we obtain
\begin{equation*}
\| \widehat{X}_k^+ - \widehat{X}_k^- + \Delta E_k \|_F^2
\lesssim 2(1+c_1)^2(1+\bar{\alpha}_k)^2
\|\mathcal{L}_k^{-1}\|_F^2 \|\mathcal{R}(\widehat{Z}_k)\|_F^2 + 2\beta_k^2 \|\widehat{X}_{k+1}\|_F^2.
\end{equation*}

Taking Frobenius norms in \eqref{final_residual_identity} and applying the triangle inequality, we obtain

\begin{equation}\label{final_residual_bound}
\begin{aligned}
\|\mathcal{R}(\widehat{Z}_{k+1})\|_F
\leq{}& \|W_k\|_F + \|\Delta\widehat{\mathcal{R}}_k\|_F + \|\mathcal{L}_k(\Delta E_k)\|_F\\
&+ \|B\|_F^2 \| \widehat{X}_k^+ - \widehat{X}_k^- + \Delta E_k \|_F^2\\
\lesssim{}& \left[ c_1(1+\bar{\alpha}_k)+\bar{\alpha}_k \right] \|\mathcal{R}(\widehat{Z}_k)\|_F\\
&+ 2(1+c_1)^2(1+\bar{\alpha}_k)^2 \|B\|_F^2 \|\mathcal{L}_k^{-1}\|_F^2 \|\mathcal{R}(\widehat{Z}_k)\|_F^2\\
&+ \beta_k \|\mathcal{L}_k\|_F \|\widehat{X}_{k+1}\|_F + 2\beta_k^2 \|B\|_F^2 \|\widehat{X}_{k+1}\|_F^2.
\end{aligned}
\end{equation}

\end{proof}

The final residual bound in Theorem \ref{thm_res} consists of four terms with distinct roles. The first term, $\rho_k\|\mathcal{R}(\widehat{Z}_k)\|_F$, is linear in the current residual and collects the effects of the low precision inner solves, precision conversion, and residual factorization errors. 
The second term, $\mu_k\|\mathcal{R}(\widehat{Z}_k)\|_F^2$, is the quadratic Newton term induced by the Riccati nonlinearity. As the residual decreases, this quadratic term becomes higher order, so the asymptotic contraction is governed mainly by $\rho_k$ and requires $\rho_k<1$, up to the remaining update error terms. 
The last two terms are proportional to $\beta_k$ and arise from truncation and finite precision formation of the updated low-rank factor. 
They contribute to the attainable residual floor, whose size is controlled by the update precision and the truncation tolerance.

\section{Complexity Analysis}
In this section, we estimate the computational complexity of Framework \ref{IR_1}. The analysis is based on the NLR-ADI iteration used for the initial CARE approximation, the LR-ADI iteration used for the inner Lyapunov correction equations, and the two low-rank compression procedures in Algorithms \ref{res_fac1} and \ref{sol_fac}. We use the notation introduced in the previous sections: $q_k=2r_k+m+p$ is the column dimension of the residual factor matrix $F_k$, and $s_k=r_k + r_k^+ + r_k^-$ is the columns dimension of the solution factors $G_k$.

For this initial NLR-ADI stage, let $k_0$ denote the number of Newton steps and let $j_0$ denote the average number of ADI shifts per Newton step. At the $i$-th Newton step, forming $A_i=A-BK_i^T$ requires $\mathcal{O}(n^2m)$ flops if $A_i$ is formed explicitly. The dominant operation in each ADI step is the solution of shifted linear systems associated with $A_i$, whose low-rank right hand side has $p+m$ columns. Let $T_{\rm shift}^{\rm init}(p+m)$ denote the cost of one such shifted solve. Then the leading cost of the initial NLR-ADI stage is
$$ T_{\rm init} = \mathcal{O}\left( k_0 j_0T_{\rm shift}^{\rm init}(p+m) \right) + T_{\rm lr}^{\rm init}, $$
where $T_{\rm lr}^{\rm init}$ collects lower order rank dependent operations, including the formation of $A_i$ and $K_i$, residual norm evaluations, and small dense matrix operations.

The residual and solution factor decompositions only involve low-rank QR factorizations and small dense spectral decompositions. For Algorithm \ref{res_fac1}, the residual factor matrix $F_k$ has $q_k$ columns. Thus the thin QR factorization and the construction of $C_k^+$ and $C_k^-$ require $\mathcal{O}(nq_k^2)$ flops, while the spectral decomposition of $H_k$ requires $\mathcal{O}(q_k^3)$ flops. Hence
$$ T_{\rm res}(k) = \mathcal{O}\left(nq_k^2+q_k^3\right). $$
For Algorithm \ref{sol_fac}, the matrix $G_k=[Z_k,Z_k^+,Z_k^-]$ has $s_k$ columns. Its QR factorization and the construction of $Z_{k+1}$ require $\mathcal{O}(n s_k^2)$ flops, while the spectral decomposition of $P_k$ requires $\mathcal{O}(s_k^3)$ flops. Therefore,
$$ T_{\rm upd}(k) = \mathcal{O}\left(n s_k^2+ s_k^3\right). $$

For the inner Lyapunov correction equations, Framework \ref{IR_1} employs LR-ADI. At refinement step $k$, the two Lyapunov equations for $X_k^+$ and $X_k^-$ share the same coefficient matrix $A_k$. Their right hand sides are given by the residual factors $C_k^+$ and $C_k^-$. Since these factors are produced from the residual factor matrix $F_k$, the combined right hand side $[C_k^+,C_k^-]$ has at most $q_k$ columns. If $j_k$ ADI shifts are used in the inner LR-ADI solves at refinement step $k$, then the dominant cost is bounded by
$$ T_{\rm inner}(k) = \mathcal{O}\left( j_k T_{\rm shift}^{\rm inner}(q_k) \right) + T_{\rm lr}^{\rm inner}(k),$$
where $T_{\rm shift}^{\rm inner}(q_k)$ denotes the cost of one shifted solve with at most $q_k$ right hand sides, and $T_{\rm lr}^{\rm inner}(k)$ collects the lower order low-rank ADI updates and residual checks.
Assuming the total number of outer iterative refinement steps is $k$, let $j$ denote the average number of ADI shifts used by the inner LR-ADI solver at each refinement step. Let $q$ and $s$ be representative upper bounds for $q_k$ and $s_k$ over the refinement process. Then the total complexity can be summarized as
$$
T_{\rm total} = \mathcal{O}\Big( k_0 j_0T_{\rm shift}^{\rm init}(p+m) + k\big[ nq^2+q^3 + jT_{\rm shift}^{\rm inner}(q_k) + ns^2+s^3 \big] \Big).
$$

The shifted linear systems in both the initial NLR-ADI stage and the inner LR-ADI correction stage involve coefficient matrices of the form $A_k=A-BK_k^T$. To avoid explicitly forming or factorizing the low-rank modified matrix $A_k+\alpha I_n$, we use the Sherman--Morrison--Woodbury(SMW) formula to simplify the action of $(A_k+\alpha I_n)^{-T}$. Let $M_\alpha=(A+\alpha I_n)^T$. Then
$$ (A_k+\alpha I_n)^{-T} = M_\alpha^{-1} + M_\alpha^{-1}K_k (I_m-B^TM_\alpha^{-1}K_k)^{-1}  B^TM_\alpha^{-1}. $$
Since $m\ll n$, the dense matrix $I_m-B^TM_\alpha^{-1}K_k$ is only of size $m\times m$, and the cost associated with this small system is lower order. Thus, after applying the SMW formula, the dominant computation is the application of $M_\alpha^{-1}$ to block right hand sides.
The cost of applying $M_\alpha^{-1}$ depends on the linear solver used to evaluate the shifted solve terms $T_{\rm shift}^{\rm init}(p+m)$ and $T_{\rm shift}^{\rm inner}(q_k)$ introduced above. For dense direct solvers, the factorization of $M_\alpha$ costs $\mathcal{O}(n^3)$ flops, and applying the factors to a block of right hand sides costs quadratically in $n$ and linearly in the number of right-hand sides. Therefore, the shifted solve terms reduce to the dense cubic scaling. For sparse direct solvers, the cost is determined by the sparse factorization of $M_\alpha$ and the subsequent triangular solves, and depends on the sparsity pattern and fill in.
For iterative solvers, the cost depends on the number of Krylov iterations and on the cost of sparse matrix vector products and preconditioner applications.

Compared with the framework of Benner et al. \cite{benner2018}, the main difference lies in the solvers used for both the initial approximation and the correction equations. Their framework computes the initial solution by a low-rank SDA method and computes the correction term by a matrix sign function method. These procedures involve inverse-type operations and dense matrix products, so the sparsity of the original coefficient matrix is difficult to preserve once dense intermediate matrices are formed. As a result, the dominant arithmetic cost is typically of cubic order, with an $\mathcal{O}(n^2)$ memory footprint.

In contrast, in Framework \ref{IR_1}, the initial approximation is computed by NLR-ADI and the correction equations are solved by LR-ADI. Both stages are organized around shifted linear solves, and the SMW formula reduces the low-rank modified systems to shifted solves with the base matrix $A+\alpha I_n$. For dense direct solvers, the shifted solve terms still lead to cubic scaling. However, for sparse direct or iterative solvers, the cost can be substantially lower when the shifted systems preserve sparsity and the fill in or iteration counts remain moderate. This makes Framework \ref{IR_1} more suitable for large-scale sparse CAREs in both computational cost and storage than methods based on dense inverse-type operations.

\section{Numerical experiments}

In all experiments, Framework \ref{IR_1} is implemented with NLR-ADI for the initial CARE approximation and LR-ADI for the inner Lyapunov correction equations. The shift parameters are generated dynamically from Krylov subspaces. The quality of the computed solution factors is assessed by the relative Frobenius norm residual of \eqref{CARE}:
$$
\text{Res} = \dfrac{\|\mathcal{R}\|_F}{2\|A\|_F\|Z^T Z\|_F + \|B^T B\|_F \|Z^T Z\|_F^2 + \|C C^T\|_F}.
$$
To ensure accuracy, the residual is computed strictly in FP64. In the practical implementation, we include a stagnation detection mechanism: specifically, if the ratio $\| \mathcal{R}(\widehat{Z}_{k+1})\|_F / \| \mathcal{R}(\widehat{Z}_{k})\|_F$ exceeds 0.999, the iteration is terminated early to avoid redundant computations.

Unless otherwise stated, the precision settings are globally fixed as $u = u_c = u_r = \text{FP64}$. To investigate the impact of mixed precision, we evaluate the algorithm under two configurations for $u_s$: FP64 (full double precision) and FP32 (mixed precision). The specific parameters in our algorithm are set as follows:
\begin{itemize}
    \item $\eta_r = 10^{-4} \cdot \max|\lambda(H)|$,
    \item $\eta_s = 10 \cdot u \cdot \max|\lambda(P)|$,
    \item $\tau_R = 10^{-14}$ and $\tau_L = 10^{-12}$ for double precision, with $\tau_L = 10^{-5}$ for single precision (where $\tau_L$ denotes the stopping tolerance for the inner Lyapunov solver),
    \item $k_{\max} = 20, \quad j_{\max} = 50$.
\end{itemize}
Here, $| \lambda(\cdot) |$ denotes the absolute value of the eigenvalues. For reproducibility, all randomly generated matrices are produced using the fixed random seed \texttt{rng(1220)}. For each test case, these matrices are generated once and kept fixed throughout 20 repeated executions. The reported computational time is averaged over these 20 runs, whereas the residuals, iteration counts, ranks, and other reported quantities are taken from the final run. In the implementation, the identity matrix $I_n$ used with $A$ is stored in the same data type as $A$ to avoid unintended precision promotion. In the following tables, ``Iter0'' and ``Iter'' denote the iteration counts in the initial phase and the refinement phase, respectively.
The counts are reported in the form ``outer(inner)'', where the value outside the parentheses gives the number of outer iterations and the value inside gives the total number of inner iterations.

\subsection{Numerical stability analysis}

In this subsection, we evaluate the numerical stability of Framework \ref{IR_1} in MATLAB R2025a. The experiments were performed on a desktop computer equipped with a 12th-generation Intel Core i7-12700 CPU running at 2.10 GHz. Motivated by the condition $d_1\kappa_F(\mathcal{L}_k)u_s<1$ in Theorem \ref{thm_res}, we construct dense CARE test problems with varying condition numbers to examine the conditioning regime in which Framework \ref{IR_1} remains effective. The shifted linear systems in both the initial NLR-ADI stage and the inner LR-ADI correction stage are solved by dense direct solvers, and the results are compared with those obtained by the MATLAB \texttt{icare} function.
\begin{example}\label{Ex1}
$$
\begin{aligned}
V &= \text{gallery}('orthog',n), \\
A &= -V * \text{diag}(\text{logspace}(0,q,n))*V^T,\\
B &= 0.2 * \text{ones}(n,m), \\
C &= 0.1 * \text{ones}(p,n).
\end{aligned}$$
\end{example}

For this dense test problem, we set $n=1000$, $m=10$, and $p=5$. By varying the parameter $q$, we test Framework \ref{IR_1} over a range of 2-norm condition numbers and compare it with MATLAB \texttt{icare}. The results are summarized in Table \ref{Tab1}.

\begin{table}[htbp]
\caption{Performance comparison on dense CARE test problems}
\label{Tab1}
\footnotesize
\centering
\setlength{\tabcolsep}{4pt}
\begin{tabular}{c l ccccc cc}
\toprule
\multirow{2}{*}{\bf $\kappa_2(A)$}  & \multirow{2}{*}{\bf Solver / } & \multicolumn{5}{c}{\bf Metrics} & \multicolumn{2}{c}{\bf Speedup} \\
\cmidrule(lr){3-7} \cmidrule(lr){8-9}
&  \bf Precision $u_s$ & Res & Iter0 & Iter & Rank & Time(s) & \bf Ourself & \texttt{\bf icare} \\
\midrule
\multirow{3}{*}{$10^1$}
& \texttt{icare} & 6.78e-15 & $-$ & $-$ & 8 & 76.92 & $-$ & $-$ \\
& Ours (FP64)    & 5.75e-16 & 13(130) & 2(16) & 8 & 2.94 & $-$ & $\mathbf{26.16 \times}$ \\
& Ours (FP32) & 6.32e-16 & 12(64) & 3(18) & 8 & 1.08 & $\mathbf{2.72 \times}$ & $\mathbf{71.22 \times}$ \\
\midrule
\multirow{3}{*}{$10^2$}
& \texttt{icare} & 6.06e-15 & $-$ & $-$ & 12 & 72.46 & $-$ & $-$ \\
& Ours (FP64)    & 3.00e-16 & 14(240) & 2(20) & 12 & 6.64 & $-$ & $\mathbf{10.91\times}$ \\
& Ours (FP32) & 3.00e-16 & 11(63) & 3(26) & 12 & 1.41 & $\mathbf{4.71\times}$ & $\mathbf{51.39\times}$ \\
\midrule
\multirow{3}{*}{$10^3$}
& \texttt{icare} & 6.34e-15 & $-$ & $-$ & 19 & 68.72 & $-$ & $-$ \\
& Ours (FP64)    & 7.37e-15 & 11(304) & 1(0) & 16 & 8.48 & $-$ & $\mathbf{8.10\times}$ \\
& Ours (FP32) & 2.46e-16 & 12(89) & 3(28) & 18 & 2.06 & $\mathbf{4.12\times}$ & $\mathbf{33.36\times}$ \\
\midrule
\multirow{3}{*}{$10^4$}
& \texttt{icare} & 5.61e-15 & $-$ & $-$ & 81 & 70.47 & $-$ & $-$ \\
& Ours (FP64)    & 9.07e-16 & 11(471) & 1(0) & 21 & 12.86 & $-$ & $\mathbf{5.48\times}$ \\
& Ours (FP32) & 3.80e-17 & 11(111) & 3(38) & 21 & 2.90 & $\mathbf{4.43\times}$ & $\mathbf{24.30\times}$ \\
\midrule
\multirow{3}{*}{$10^5$}
& \texttt{icare} & 5.30e-15 & $-$ & $-$ & 208 & 67.65 & $-$ & $-$ \\
& Ours (FP64)    & 1.10e-16 & 12(600) & 2(19) & 25 & 16.34 & $-$ & $\mathbf{4.14\times}$ \\
& Ours (FP32) & 5.54e-17 & 11(153) & 3(52) & 38 & 4.09 & $\mathbf{4.00\times}$ & $\mathbf{16.54\times}$ \\
\midrule
\multirow{3}{*}{$10^6$}
& \texttt{icare} & 5.09e-15 & $-$ & $-$ & 341 & 64.69 & $-$ & $-$ \\
& Ours (FP64)    & 1.95e-15 & 11(550) & 3(58) & 27 & 16.07 & $-$ & $\mathbf{4.03\times}$ \\
& Ours (FP32) & 5.77e-17 & 11(379) & 4(96) & 60 & 12.85 & $\mathbf{1.25\times}$ & $\mathbf{5.03\times}$ \\
\midrule
\multirow{3}{*}{$10^7$}
& \texttt{icare} & 5.52e-15 & $-$ & $-$ & 414 & 60.83 & $-$ & $-$ \\
& Ours (FP64)    & 2.12e-15 & 10(500) & 3(53) & 30 & 14.92 & $-$ & $\mathbf{4.08\times}$ \\
& Ours (FP32) & 7.07e-16 & 11(448) & 4(123) & 175 & 14.00 & $\mathbf{1.07\times}$ & $\mathbf{4.35\times}$ \\
\midrule
\multirow{3}{*}{$10^8$}
& \texttt{icare} & 4.82e-15 & $-$ & $-$ & 534 & 60.52 & $-$ & $-$ \\
& Ours (FP64)    & 8.29e-17 & 13(650) & 2(50) & 32 & 17.21 & $-$ & $\mathbf{3.52\times}$ \\
& Ours (FP32) & $-$ & $-$ & $-$ & $-$ & $-$ & $-$ & $-$ \\
\bottomrule
\end{tabular}
\end{table}

Table \ref{Tab1} shows that, when $\kappa_2(A)\leq 10^7$, Framework \ref{IR_1} attains residuals comparable to those of \texttt{icare} in both FP64 and FP32. This indicates that the mixed precision iterative refinement process can maintain numerical stability within a certain range of condition numbers. When $\kappa_2(A)=10^8$, however, the FP32 variant fails to converge, showing that severe ill-conditioning can compromise the effectiveness of the low-precision inner solves.

In terms of runtime, Framework \ref{IR_1} is substantially faster than \texttt{icare}. In FP64, the speedup over \texttt{icare} ranges from approximately $3.52\times$ to $26.16\times$. In FP32, whenever convergence is achieved, the speedup over \texttt{icare} ranges from approximately $4.35\times$ to $71.22\times$. Moreover, FP32 further reduces the runtime relative to FP64, with speedups ranging from $2.72\times$ to $4.71\times$ when $\kappa_2(A)\leq 10^5$. This demonstrates that low precision inner solves can significantly improve computational efficiency for moderately conditioned problems.

As the condition number further increases, the speedup obtained by FP32 decreases. When $\kappa_2(A)=10^6$ and $10^7$, the FP32 speedups over FP64 drop to $1.25\times$ and $1.07\times$, respectively, while the corresponding ranks increase to 60 and 175. By contrast, the ranks of the FP64 variant remain no larger than 32, whereas the ranks associated with \texttt{icare} increase from 8 to 534 as the condition number grows. These results show that the proposed low-rank framework can effectively control the dimension of the solution factor. However, once $\kappa_F(A_k)$ exceeds the effective range of single precision, low precision errors reduce the efficiency of the FP32 corrections and lead to rank growth. Therefore, for highly ill-conditioned problems, low precision inner solves should be avoided unless appropriate preconditioning is introduced.

\subsection{Performance Analysis on GPU}
In this subsection, we evaluate the performance of the Framework \ref{IR_1} under different precision configurations. For the shifted linear systems arising in NLR-ADI and LR-ADI, we use restarted GMRES preconditioned by ILU(0). The ILU(0) factorization and triangular solves are performed on the GPU using the csrilu02 and SpSM routines from cuSPARSE, respectively. The GMRES parameters are chosen according to the inner working precision. For FP32 inner solves, the restart length, maximum number of iterations, and stopping tolerance are 30, 100, and $10^{-5}$, respectively; for FP64 inner solves, the corresponding values are 50, 200, and $10^{-10}$. We also test the framework proposed in \cite{benner2018} as a baseline for comparison.

Both frameworks are implemented with GPU acceleration. The computationally dominant large-scale linear algebra operations, including dense and sparse matrix operations, linear system solves, orthogonalization, residual evaluation, and low-rank factor updates, are mainly performed on the GPU using cuBLAS, cuSPARSE, cuSOLVER, and custom CUDA kernels. The CPU handles data loading, problem initialization, algorithm control, convergence checks, shift selection based on small projected eigenvalue problems, and a small number of scalar or index-based operations. The reported runtimes include CPU computations, GPU computations, and host device data transfer overhead.

All experiments were conducted on a single NVIDIA Tesla V100 PCIe GPU with 32 GB memory. The host machine was equipped with two Intel(R) Xeon(R) Silver 4214 CPUs running at 2.20 GHz, 128 GB RAM, and Ubuntu 22.04.5 LTS. The code was compiled with nvcc from CUDA Toolkit 12.8.

To evaluate the performance on realistic sparse structures, we construct test CAREs using matrices from the SuiteSparse Matrix Collection \cite{UF-Sparse-Matrix-Collection}. For each selected matrix $M$, we set $A=-M$, and the corresponding matrix properties are summarized in Table \ref{juz}. The matrices $B \in \mathbb{R}^{n \times 3}$ and $C \in \mathbb{R}^{2 \times n}$ are generated randomly, with $B$ normalized to satisfy $\|B\|_F=1$. The performance comparisons based on these test problems are reported in Table \ref{Ex3_Tab}.

\begin{table}[htbp]
\caption{Properties of the SuiteSparse test matrices}
\label{juz}
\footnotesize
      \begin{center}
              \begin{tabular}{c c c c c c c}
                      \hline
          Name & $n$  & $\kappa_F$  & SPD & $\min(\lambda)$ & $\max(\lambda)$ & Nonzeros \\ \hline
    \texttt{cage9}    & 3534  & 4.19e+03 & No  & 9.66e-02 & 1.00e+00 & 41,594  \\
    \texttt{fv1}         & 9604  & 1.54e+04 & Yes & 5.12e-01 &4.51e+00 & 85,264  \\
    \texttt{cage10}      & 11397 & 1.30e+04 & No  & 9.94e-02 & 1.00e+00 & 150,640 \\
     \texttt{poli\_large}  & 15575 & 4.07e+04 & No  & 5.87e-02 & 1.94e+00 & 33,033  \\
     \texttt{poli3}       & 16955 & 3.60e+04 & No  & 4.62e-01 & 1.52e+00 & 37,849  \\
    \texttt{FEM\_3D\_thermal1} & 17880 & 2.06e+05 & No  & 7.00e-03 & 5.35e+00 & 430,740 \\
    \texttt{poli4}    & 33833  & 7.86e+04 & No & 2.06e-01 & 1.71e+00 & 73,249 \\
             \hline
              \end{tabular}
      \end{center}
\end{table}

\begin{table}[htbp]
\caption{GPU performance on SuiteSparse test problems}
\label{Ex3_Tab}
\footnotesize
\centering
\setlength{\tabcolsep}{2pt}
\begin{tabular}{c l cccc cccc c}
\toprule
\multirow{2}{*}{\bf Name} & \multirow{2}{*}{\bf Method} & \multicolumn{4}{c}{$u_s = \text{FP}64$} & \multicolumn{4}{c}{$u_s = \text{FP}32$} & \multirow{2}{*}{\bf Speedup} \\
\cmidrule(lr){3-6} \cmidrule(lr){7-10}
& & Res & Iter0 & Iter & Time(s) & Res & Iter0 & Iter &  Time(s) & \\
\midrule
\multirow{3}{*}{\texttt{cage9}}
& Benner & 1.79e-17 & 9     & 1(0)  & 0.37 & 4.21e-15 & 8     & 6(15) & 0.85 & $0.44\times$ \\
& Ours & 9.62e-19 & 6(79) & 2(14)  & 1.39 & 4.16e-17 & 6(34) & 3(17)  & 0.69 & $\mathbf{2.02\times}$ \\
& {\bf Speedup} & $-$        & $-$   & $-$& $0.27\times$ & $-$ & $-$ & $-$ & $\mathbf{1.23\times}$ & $-$ \\
\midrule
\multirow{3}{*}{\texttt{fv1}}
& Benner & 3.54e-17 & 8     & 1(0)  & 3.03 & 4.75e-15 & 7     & 5(12) & 4.96 & $0.61\times$ \\
& Ours & 8.99e-15 & 4(51) & 1(0)  & 1.72 & 9.81e-17 & 4(21) & 3(14) & 1.50 & $\mathbf{1.15\times}$ \\
& {\bf Speedup} & $-$     & $-$   & $-$& $\mathbf{1.76\times}$ & $-$ & $-$ & $-$ & $\mathbf{3.31\times}$ & $-$ \\
\midrule
\multirow{3}{*}{\texttt{cage10}}
& Benner & 7.42e-18 & 9     & 1(0)  & 5.31 & 5.52e-15 & 8   & 6(15) & 9.18 & $0.58 \times$ \\
& Ours & 2.40e-15 & 6(82) & 1(0)   & 1.20 & 5.18e-18 & 6(26) & 3(15)  & 0.76 & $\mathbf{1.57\times}$ \\
& {\bf Speedup} & $-$      & $-$     & $-$& $\mathbf{4.42\times}$ & $-$ & $-$ & $-$ & $\mathbf{12.00\times}$ & $-$ \\
\midrule
\multirow{3}{*}{\texttt{poli\_large}}
& Benner & 9.32e-17 & 19    & 1(0)   & 24.49 & 9.51e-15 & 18    & 17(48) & 59.77 & $ 0.41\times$ \\
& Ours & 1.48e-15 & 6(112)& 1(0)   & 1.78  & 1.10e-16 & 6(72) & 3(19)  & 1.28  & $\mathbf{1.39\times}$ \\
& {\bf Speedup} & $-$      & $-$      & $-$& $\mathbf{13.80\times}$ & $-$ & $-$  & $-$ & $\mathbf{46.84\times}$ & $-$ \\
\midrule
\multirow{3}{*}{\texttt{poli3}}
& Benner & 9.63e-17 & 16    & 1(0)  & 26.87 & 6.34e-15 & 15  & 3(6)  & 20.65 & $1.30\times$ \\
& Ours & 1.51e-16 & 6(88) & 1(0)   & 1.38  & 6.07e-18 & 6(37) & 3(14) & 0.93  & $\mathbf{1.49 \times}$\\
& {\bf Speedup}  & $-$       & $-$   & $-$& $\mathbf{19.41\times}$ & $-$  & $-$ & $-$ & $\mathbf{22.19\times}$ & $-$ \\
\midrule
\multirow{3}{*}{\shortstack[c]{\texttt{FEM\_3D\_}\\\texttt{thermal1}}}
& Benner & 6.70e-18 & 12     & 1(0)  & 24.41 & 1.63e-12 & 10     & 20(60) & 97.47 & $0.25\times$ \\
& Ours & 8.33e-15 & 8(224) & 1(0) & 14.60 & 2.06e-16 & 8(115) & 3(32) & 9.07  & $\mathbf{1.61\times}$ \\
& {\bf Speedup}  & $-$        & $-$   & $-$& $\mathbf{1.67\times}$ & $-$ & $-$  & $-$ & $\mathbf{10.75\times}$ & $-$ \\
\midrule
\multirow{3}{*}{\texttt{poli4}}
& Benner & $-$ & $-$    & $-$  & $-$ & $-$ & $-$  & $-$ & $-$ & $-$ \\
& Ours & 7.57e-15 & 6(110) & 1(0) & 2.34 & 3.34e-18 & 6(48) & 3(17) & 1.36  & $\mathbf{1.72\times}$ \\
& {\bf Speedup}  & $-$        & $-$   & $-$& $-$ & $-$ & $-$  & $-$ & $-$ & $-$ \\
\bottomrule
\end{tabular}
\end{table}

The results in Table \ref{Ex3_Tab} indicate that both methods attain small and comparable residuals in the FP64 runs, while Framework \ref{IR_1} continues to deliver accurate solutions when FP32 is used for the inner solves. In contrast, Benner et al.'s method exhibits noticeably larger residuals for several test problems in FP32. In terms of runtime, Framework \ref{IR_1} is faster than Benner et al.'s method for most test problems.
With FP64 inner solves, except for the case constructed from the relatively small scale matrix \texttt{cage9}, Framework \ref{IR_1} achieves speedups ranging from approximately $1.67\times$ to $19.41\times$. When the inner solves are performed in FP32, Framework \ref{IR_1} is faster for all test problems for which Benner et al.'s method is executable, with speedups ranging from $1.23\times$ to $46.84\times$. Moreover, Framework \ref{IR_1} benefits consistently from the use of FP32 inner solves, with speedups between $1.15\times$ and $2.02\times$ over its FP64 counterpart.
For the larger test problem constructed from \texttt{poli4}, Benner et al.'s method runs out of GPU memory, whereas Framework \ref{IR_1} remains executable and achieves a $1.72\times$ speedup from FP32 inner solves. Overall, these observations indicate that the proposed framework exploits sparsity more effectively, has lower memory requirements, and provides more reliable acceleration in mixed precision.

The observed performance difference can be attributed mainly to the different choices of the initial solver and the inner Lyapunov solver in the two Newton-type frameworks. Benner et al.'s method uses LR-SDA as the initial solver and the matrix sign function method as the inner Lyapunov solver, whereas Framework \ref{IR_1} uses the NLR-ADI and LR-ADI algorithms, respectively. The matrix sign function method involves dense matrix operations and tends to lose the benefits of the original sparse structure, leading to higher memory usage and computational cost for large-scale sparse problems.
In contrast, the dominant computations in NLR-ADI/LR-ADI are shifted linear solves, which can be performed by preconditioned GMRES while preserving and exploiting sparsity. Therefore, Framework \ref{IR_1} is more suitable for large-scale sparse CAREs.

\begin{example}
We consider the following artificial test problem constructed from a Toeplitz matrix:
$$
A =
\left[\begin{array}{cccccc}
-2.8 & -1 & -1 & -1 & 0 & \cdots \\
1 & -2.8 & -1 & -1 & -1 & \ddots \\
0 & 1 & -2.8 & -1 & -1 & \ddots \\
\vdots & \ddots & \ddots & \ddots & \ddots & \ddots \\
0 & \cdots & 0 & 1 & -2.8 & -1 \\
0 & \cdots & 0 & 0 & 1 & -2.8
\end{array}\right]_{n\times n}.
$$
and $B\in \mathbb{R}^{n \times 10}$, $C\in \mathbb{R}^{5 \times n}$ are random matrices, with $\|B\|_F=1$.
\end{example}

We vary the matrix dimension $n$ to evaluate the accuracy stability, iteration behavior, and runtime growth of Framework \ref{IR_1} for large-scale problems. The results are reported in Table \ref{tab:toeplitz}.

\begin{table}[htbp]
\caption{GPU performance on Toeplitz test problems}
\label{tab:toeplitz}
\footnotesize
\centering
\setlength{\tabcolsep}{2pt}
\begin{tabular}{c cccc cccc c}
\toprule
\multirow{2}{*}{\bf $n$}
& \multicolumn{4}{c}{$u_s = \text{FP}64$}
& \multicolumn{4}{c}{$u_s = \text{FP}32$}
& \multirow{2}{*}{\bf Speedup} \\
\cmidrule(lr){2-5} \cmidrule(lr){6-9}
& Res & Iter0 & Iter  & Time(s)
& Res & Iter0 & Iter  & Time(s) & \\
\midrule
32768  & 2.64e-16 & 4(128) & 1(0) & 24.64 & 4.46e-17 & 4(40) & 3(26) & 16.01 & $\mathbf{1.54\times}$ \\
49152  & 3.83e-16 & 4(125) & 1(0) & 35.58 & 2.24e-17 & 4(40) & 3(26) & 23.76 & $\mathbf{1.50\times}$ \\
65536  & 1.51e-16 & 4(128) & 1(0) & 48.18 & 2.32e-17 & 4(40) & 3(26) & 31.68 & $\mathbf{1.52\times}$ \\
81920  & 2.20e-16 & 4(125) & 1(0) & 58.17 & 1.11e-17 & 4(40) & 3(26) & 39.76 & $\mathbf{1.46\times}$ \\
98304  & 1.86e-16 & 4(125) & 1(0) & 69.46 & 1.47e-17 & 4(40) & 3(26) & 47.56 & $\mathbf{1.46\times}$ \\
114688 & 1.77e-16 & 4(125) & 1(0) & 80.79 & 1.11e-17 & 4(40) & 3(27) & 56.46 & $\mathbf{1.43\times}$ \\
131072 & 1.64e-16 & 4(125) & 1(0) & 92.07 & 9.32e-18 & 4(40) & 3(26) & 63.44 & $\mathbf{1.45\times}$ \\
\bottomrule
\end{tabular}
\end{table}

As shown in Table \ref{tab:toeplitz}, the residuals and iteration counts remain stable under both FP64 and FP32 as the matrix dimension $n$ increases, indicating that the proposed framework maintains good accuracy and convergence behavior for this structured test problem. Meanwhile, the runtime increases gradually with $n$, while FP32 is consistently faster than FP64 for all tested dimensions, with speedups of about $1.5\times$ across all tested dimensions.
These results show that mixed precision inner solves provide a moderate but stable computational benefit without compromising the stability of the algorithm.

\section{Conclusions}

In this paper, we have proposed a Newton-based mixed precision iterative refinement framework for large-scale sparse continuous-time algebraic Riccati equations. The framework computes the initial approximation and the inner Lyapunov correction equations in lower precision, while residual evaluation and solution updates are performed in higher precision. The low-rank factor decomposition and truncation strategies enable indefinite residuals and Newton correction terms to be treated in factored form, while maintaining positive semidefiniteness, controlling rank growth, and avoiding the explicit formation of large-scale dense matrices.
A first-order rounding error analysis relates the stability of the mixed precision refinement process to the conditioning of the Lyapunov operator and the unit roundoff of the lower precision inner solves. 
Numerical experiments confirm this conditioning effect and show that, for problems in the admissible conditioning regime, the proposed framework achieves clear speedups over the full double precision implementation while maintaining comparable accuracy. For large-scale sparse test problems, the ADI-based realization exploits sparsity effectively and demonstrates good computational efficiency on GPU platforms. 

Future work will investigate the implementation and optimization of the proposed framework on hardware platforms with native support for lower precision formats, such as FP16 and Bfloat16, as well as acceleration units such as Tensor Cores. Another direction is to investigate solver selection strategies for different classes of coefficient matrices, including the choice of initial solvers, inner Lyapunov solvers, and preconditioners for high condition number problems.

\section*{Acknowledgments}
The numerical experiments were performed on the High Performance Computing Platform of Xiangtan University. The authors gratefully acknowledge its computational resources and technical support.

\bibliographystyle{siamplain}
\bibliography{references}

\begin{thebibliography}{10}

\bibitem{mxiedzongsu2021}
{\sc A.~Abdelfattah, H.~Anzt, A.~Ayala, E.~Boman, E.~Carson, S.~Cayrols, T.~Cojean, J.~Dongarra, R.~Falgout, M.~Gates, et~al.}, {\em Advances in mixed precision algorithms: 2021 edition}, Tech. Report LLNL-TR-825909, Lawrence Livermore National Laboratory (LLNL), Livermore, CA (United States), 2021, \url{https://doi.org/10.2172/1814677}.

\bibitem{mxjj2005}
{\sc A.~Antoulas}, {\em Approximation of Large-Scale Dynamical Systems}, vol.~6, Society for Industrial and Applied Mathematics, 2005, \url{https://doi.org/10.1137/1.9780898718713}.

\bibitem{sign1998}
{\sc Z.~Bai and J.~Demmel}, {\em Using the matrix sign function to compute invariant subspaces}, SIAM Journal on Matrix Analysis and Applications, 19 (1998), pp.~205--225, \url{https://doi.org/10.1137/S0895479896297719}.

\bibitem{radi2018}
{\sc P.~Benner, Z.~Bujanovi{\'c}, P.~K{\"u}rschner, and J.~Saak}, {\em Radi: a low-rank adi-type algorithm for large scale algebraic riccati equations}, Numerische Mathematik, 138 (2018), pp.~301--330, \url{https://doi.org/10.1007/s00211-017-0907-5}.

\bibitem{benner2017}
{\sc P.~Benner, E.~Dufrechou, P.~Ezzatti, and A.~Rem{\'o}n}, {\em Studying mixed precision techniques for the solution of algebraic riccati equations}, in 2nd Workshop on Power-Aware Computing, Ringberg Castle, 2017, pp.~1--6, \url{https://doi.org/10.5281/zenodo.815496}.

\bibitem{benner2018}
{\sc P.~Benner, E.~Dufrechou, P.~Ezzatti, A.~Rem{\'o}n, and J.~Saak}, {\em A gpu-aware mixed-precision solver for low-rank algebraic riccati equations}, Concurrency and Computation: Practice and Experience, 31 (2019), p.~e4462, \url{https://doi.org/10.1002/cpe.4462}.

\bibitem{INRADI2016}
{\sc P.~Benner, M.~Heinkenschloss, J.~Saak, and H.~K. Weichelt}, {\em An inexact low-rank newton--adi method for large-scale algebraic riccati equations}, Applied Numerical Mathematics, 108 (2016), pp.~125--142, \url{https://doi.org/10.1016/j.apnum.2016.05.006}.

\bibitem{liu2025lyapunov}
{\sc P.~Benner and X.~Liu}, {\em Mixed-precision iterative refinement for low-rank lyapunov equations}, 2025, \url{https://arxiv.org/abs/2510.02126}.

\bibitem{Krylov2024}
{\sc C.~Bertram and H.~Fa{\ss}bender}, {\em A class of petrov--galerkin krylov methods for algebraic riccati equations}, Electronic Transactions on Numerical Analysis, 62 (2024), pp.~138--162, \url{https://doi.org/10.1553/etna_vol62s138}.

\bibitem{mpqr2025}
{\sc A.~Buttari, T.~Mary, and A.~Pacteau}, {\em Truncated qr factorization with pivoting in mixed precision}, SIAM Journal on Scientific Computing, 47 (2025), pp.~B382--B401, \url{https://doi.org/10.1137/24M1644705}.

\bibitem{carson2018}
{\sc E.~Carson and N.~J. Higham}, {\em Accelerating the solution of linear systems by iterative refinement in three precisions}, SIAM Journal on Scientific Computing, 40 (2018), pp.~A817--A847, \url{https://doi.org/10.1137/17M1140819}.

\bibitem{carson2020ls}
{\sc E.~Carson, N.~J. Higham, and S.~Pranesh}, {\em Three-precision gmres-based iterative refinement for least squares problems}, SIAM Journal on Scientific Computing, 42 (2020), pp.~A4063--A4083, \url{https://doi.org/10.1137/20M1316822}.

\bibitem{chu2005doubling}
{\sc E.~K.-W. Chu, H.-Y. Fan, and W.-W. Lin}, {\em A structure-preserving doubling algorithm for continuous-time algebraic riccati equations}, Linear Algebra and its Applications, 396 (2005), pp.~55--80, \url{https://doi.org/10.1016/j.laa.2004.10.010}.

\bibitem{UF-Sparse-Matrix-Collection}
{\sc T.~A. Davis and Y.~Hu}, {\em The university of florida sparse matrix collection}, ACM Transactions on Mathematical Software, 38 (2011), pp.~1--25, \url{https://doi.org/10.1145/2049662.2049663}.

\bibitem{liu2025sylvester}
{\sc A.~Dmytryshyn, M.~Fasi, N.~J. Higham, and X.~Liu}, {\em Mixed-precision algorithms for solving the sylvester matrix equation}, 2025, \url{https://arxiv.org/abs/2503.03456}.

\bibitem{EbArnoldi2009}
{\sc M.~Heyouni and K.~Jbilou}, {\em An extended block arnoldi algorithm for large-scale solutions of the continuous-time algebraic riccati equation}, Electronic Transactions on Numerical Analysis, 33 (2009), pp.~53--62.

\bibitem{mxiedzongsu2022}
{\sc N.~J. Higham and T.~Mary}, {\em Mixed precision algorithms in numerical linear algebra}, Acta Numerica, 31 (2022), pp.~347--414, \url{https://doi.org/10.1017/S0962492922000022}.

\bibitem{IEEE}
{\sc {IEEE}}, {\em {IEEE Standard for Floating-Point Arithmetic}}.
\newblock {IEEE Std 754-2019 (Revision of IEEE 754-2008)}, July 2019, \url{https://doi.org/10.1109/IEEESTD.2019.8766229}.

\bibitem{mxjj1983}
{\sc E.~Jonckheere and L.~Silverman}, {\em A new set of invariants for linear systems--application to reduced order compensator design}, IEEE Transactions on Automatic Control, 28 (1983), pp.~953--964, \url{https://doi.org/10.1109/TAC.1983.1103159}.

\bibitem{dizhi2016}
{\sc P.~K{\"u}rschner}, {\em Efficient low-rank solution of large-scale matrix equations}, PhD thesis, Otto-von-Guericke-Universit{\"a}t Magdeburg, Aachen, 2016, \url{https://www.shaker.eu/shop/978-3-8440-4385-3}.

\bibitem{schur1979}
{\sc A.~Laub}, {\em A schur method for solving algebraic riccati equations}, IEEE Transactions on Automatic Control, 24 (1979), pp.~913--921, \url{https://doi.org/10.1109/TAC.1979.1102178}.

\bibitem{subspace2015}
{\sc Y.~Lin and V.~Simoncini}, {\em A new subspace iteration method for the algebraic riccati equation}, Numerical Linear Algebra with Applications, 22 (2015), pp.~26--47, \url{https://doi.org/10.1002/nla.1936}.

\bibitem{zuiyoukz2001}
{\sc A.~Locatelli}, {\em Optimal Control: An Introduction}, Birkh\"{a}user Basel, 1~ed., 2001.

\bibitem{projectNK2019}
{\sc D.~Palitta}, {\em The projected newton--kleinman method for the algebraic riccati equation}, 2019, \url{https://arxiv.org/abs/1901.10199}.

\bibitem{prikopa2013}
{\sc K.~E. Prikopa and W.~N. Gansterer}, {\em On mixed precision iterative refinement for eigenvalue problems}, Procedia Computer Science, 18 (2013), pp.~2647--2650, \url{https://doi.org/10.1016/j.procs.2013.06.002}.

\bibitem{sign1980}
{\sc J.~D. Roberts}, {\em Linear model reduction and solution of the algebraic riccati equation by use of the sign function}, International Journal of Control, 32 (1980), pp.~677--687, \url{https://doi.org/10.1080/00207178008922881}.

\bibitem{arXiv2026Lyapunov}
{\sc J.~Schulze and J.~Saak}, {\em Toward a mixed-precision adi method for lyapunov equations}, Proceedings in Applied Mathematics and Mechanics, 26 (2026), p.~e70145, \url{https://doi.org/10.1002/pamm.70145}.

\bibitem{zhang2024lowrank}
{\sc J.~Zhang and W.~Xun}, {\em Low-rank generalized alternating direction implicit iteration method for solving matrix equations}, Computational and Applied Mathematics, 43 (2024), p.~256, \url{https://doi.org/10.1007/s40314-024-02774-8}.

\end{thebibliography}

\end{document}